\documentclass[12pt, reqno]{amsart}
\usepackage{enumerate}
\usepackage{amsmath, amsthm, amscd, amsfonts, amssymb, graphicx, color}
\usepackage[bookmarksnumbered, colorlinks, plainpages]{hyperref}
\input{mathrsfs.sty}

\newtheorem{theorem}{Theorem}[section]
\newtheorem{lemma}[theorem]{Lemma}

\newtheorem{corollary}[theorem]{Corollary}
\theoremstyle{definition}

\theoremstyle{remark}

\numberwithin{equation}{section}

\begin{document}

\title[ A generalization of Powers-St{\o}rmer  Inequality ] {A generalization of Powers-St{\o}rmer  Inequality }
\author[Anchal Aggarwal and Mandeep Singh]{Anchal Aggarwal and Mandeep Singh}
\address{Department of Mathematics, Sant Longowal Institute of Engineering and
Technology, Longowal-148106, Punjab, India} \email{anchal8692@gmail.com}\email{msrawla@yahoo.com} 

\begin{abstract} Let $A,\;B$ be the positive semidefinite matrices. A matrix version of the famous Powers-St{\o}rmer's  inequality
$$2Tr(A^\alpha B^{1-\alpha})\geq Tr(A+B-|A-B|),\;\;\;0\leq\alpha\leq 1,$$
was proven by Audenaert et. al. We establish a comparison of eigenvalues for the matrices $A^\alpha B^{1-\alpha}$ and $A+B-|A-B|, \; 0 \leq \alpha \leq 1,$ subsuming the Powers-St{\o}rmer's  inequality.  We also prove several related norm inequalities.
\end{abstract}
\maketitle

\section{Introduction}
Let $M_n$ denote the algebra of all $n \times n$ complex matrices. A Hermitian member $A$ of $M_n$ with all non-negative eigenvalues is known as positive semi-definite matrix, simply denoted by $A\geq 0.$ We shall denote by $P_n,$ the collection of all such matrices. For $A,\;B$ Hermitian in $M_n,$ we employ the positive semi-definite ordering: $A\geq B$ if and only if $A-B\geq 0.$ By $|A|,$ we mean the positive square root of the matrix $A^*A,$ i.e., $(A^*A)^{1/2}.$ The Jordan decomposition of a Hermitian matrix $A$ is given by $A = A_+ - A_-, $ where $A_+$ and $A_-$ are the members of the $P_n$ along with $A_+ A_- =0$ (see  \cite{B}, page 99).
We shall consider
$\lambda_{1}(A)\geq\lambda_{2}(A)\geq\cdots\lambda_{n}(A)\geq 0,$ the eigenvalues of  $A\in P_n,$ arranged in
decreasing order and repeated according to their multiplicity.
Similarly $s_{1}(A)\geq s_{2}(A)\geq\cdots s_{n}(A)\geq 0,$ denote the singular values (eigenvalues of $|A|$) of a matrix $A\in P_n,$ arranged in
decreasing order and repeated according to their multiplicity. By $|||.|||,$ we mean any unitarily invariant norm, while $||. ||$ denotes operator norm on    $M_n.$

In 2007, Audenaert et. al. \cite{KM} solved a long standing open problem to identify
the classical quantum Chernoff bound in the area of information theory.
After the mathematical formulation of that problem, they proved a nontrivial
and fundamental inequality relating to the trace distance to the quantum Chernoff bound.
That became a key result to a solution of the problem and is stated as follows:\\
Let $A, \; B$ be positive matrices and $0\leq \alpha\leq 1.$ Then
\begin{equation}\label{E1} 2Tr (A^\alpha B^{1-\alpha})\geq Tr(A+B-|A-B|)\end{equation} holds. A particular case $\alpha=1/2$ in \eqref{E1} is a well known Powers-St{\o}rmer's  inequality \cite{PS}, which was proved in 1970. For such literature and detail of inequalities the reader may refer \cite{P}.
Subsequently in 2008, again Audenaert et. al. \cite{KMR} worked on symmetric as well as with asymmetric quantum
hypothesis testing. In \cite{KMR}  also, they proved some similar  type of  inequalities as that of \eqref{E1} which played a key role in getting the optimal solution to the symmetric classical hypothesis test.

In 2011, Y. Ogata \cite{YO} generalised the Powers-St{\o}rmer inequality to von Neumann algebras. Recently several authors including D. Hoa et. al. \cite{DO} generalised this inequality on $C^*-$algebras using the technique of operator monotone functions on $[0,\infty).$

We aim to prove the comparison of eigenvalues of $A+B-|A-B|$ and $2A^\alpha B^{1-\alpha},$ generalizing all the forms of Powers-St{\o}rmer's inequality. We shall also prove several other associated norm inequalities.

\section{Main Results}

\begin{lemma}\label{L1} Let $A ,\;B \in P_n$ then there exist a  matrix $S \in P_n$ satisfying
\begin{enumerate}[\rm{(}1{\rm)}]
              \item $S\leq A,\;S\leq B$
              \item if $T\leq A,\;T\leq B,$  is a fixed Hermitian matrix then $\lambda_{i}(T)\leq \lambda_{i}(S)$ for  $1\leq i\leq n.$
            \end{enumerate}
\end{lemma}
\begin{proof} We shall first prove this result for  either of  $A$ or $ B$  invertible. So assume $B$ is invertible i.e. $B$ is Hermitian and whose all the eigenvalues are positive.  As is well-known that $B^{-1/2}AB^{-1/2}\in P_{n}$ and so unitarily diagonalizable. We assume that $B^{-1/2}AB^{-1/2}=U^{*}DU$ for some $U$ a unitary and $D$ a diagonal matrix with diagonal entries as $d_1\geq d_2\geq d_3\cdots\geq d_n\geq 0.$
Choose $S=B^{1/2}U^*D_1UB^{1/2},$ where $D_1$ is a diagonal matrix with diagonal entries as $t_1\geq t_2\geq t_3\cdots\geq t_n\geq 0,$ such that $t_i=min\{d_i, 1\}.$ This choice of $S$ satisfies
$$S=B^{1/2}U^*D_1UB^{1/2}\leq B^{1/2}U^*DUB^{1/2}= A,$$
$$S=B^{1/2}U^*D_1UB^{1/2}\leq B^{1/2}U^*IUB^{1/2}= B.$$
For (2), let $T\leq A$ as well as $T\leq B$ be a fixed Hermitian matrix, then by Weyl's monotonicity principle we have
$\lambda_{i}(T)\leq \lambda_{i}(A)$ and 
$\lambda_{i}(T)\leq \lambda_{i}(B)$ for all $i=1,2,\cdots,n.$
If 
$$\lambda_{i}(T)\leq \lambda_{i}(S)\;\;\; for \;\; 1\leq i\leq n,$$ the above construction of S meets both the requirements.\\
If
$$\lambda_{j}(T)\geq \lambda_{j}(S)\;\;\; for\;some \;\; 1\leq j\leq n,$$ then we replace that particular $t_j$ with $\lambda_j(T)$ in $D_1.$ 
Then, this choice of $S$ meets both the requirements.\\
The general case follows by using continuity argument.
\end{proof}

Now onwards, we shall   denote $S$ by $min \{A,B\}.$

\begin{theorem}\label{T1} Let $A ,\;B \in P_n$ then
\begin{equation}\label{E2} \lambda_i(A+B-|A-B|)\leq 2\lambda_i(A^{\alpha} B^{1-\alpha}) \end{equation}
for  $0\leq \alpha \leq  1\;$ and $1\leq i\leq n.$

\end{theorem}
\begin{proof} Let $T$ be any Hermitian matrix with Jordan decomposition $T_+ - T_-.$ Then, $|T|= T_+ + T_-,$ so $T-|T|=-2T_-\leq 0.$
Using this fact for $A-B,$ we can write,
\begin{eqnarray}\label{E3}A+B-|A-B| =2(B-(A-B)_{-})\leq 2B.
\end{eqnarray}
Replacing $B$ by $A$ in above inequality, we obtain
\begin{eqnarray}\label{E4}A+B-|A-B| =2(A-(B-A)_{-})\leq 2A.
\end{eqnarray}
Now, on using Lemma \ref{L1}, we obtain
\begin{eqnarray}\label{E5}\lambda_i(A+B-|A-B|) &\leq & 2 \lambda_i(\;min \{A,B\}=S)\nonumber\\
&\leq&  2\lambda_i(S^{\alpha /2}B^{1-\alpha}S^{\alpha/2}),\;\;\;for\;\; 1\leq i\leq n.\nonumber\end{eqnarray}
To complete the proof, it is enough to show
 $$\lambda_i(S^{\alpha /2}B^{1-\alpha}S^{\alpha/2})\leq\lambda_i(A^{\alpha}B^{1-\alpha}), \;\;\;for\;\; 1\leq i\leq n.$$
Indeed,\begin{eqnarray}\label{E6} \;\;\;
 2\lambda_i(S^{\alpha /2}B^{1-\alpha}S^{\alpha/2})
&=&2\lambda_i(B^{(1-\alpha)/2}S^{\alpha}B^{(1-\alpha)/2})\nonumber\\
&\leq& 2\lambda_i(B^{(1-\alpha)/2}A^{\alpha}B^{(1-\alpha)/2})
\nonumber\\
&=& 2\lambda_i(A^{\alpha}B^{1-\alpha})\;\;\;for\;\; 1\leq i\leq n.\nonumber
\end{eqnarray}
\end{proof}

\begin{corollary}\label{C1}\rm{(}Cf. \cite{KM,KMR},Theorem 1,Theorem 2 \rm{)} Let $A ,\;B \in P_n$ then for  $0\leq \alpha \leq  1$
\begin{equation}\label{E7} 0\leq Tr(A+B-|A-B|)\leq 2 Tr(A^{\alpha} B^{1-\alpha}). \end{equation}
\end{corollary}
\begin{proof} Let $A-B=(A-B)_+-(A-B)_-$ be the Jordan decomposition  of $A-B,$ then for $1 \leq i \leq n,$
\begin{equation}\label{E11}\lambda_i{(A-B)_-}\leq \lambda_i(B),\end{equation}
(see Lemma IX.4.1 of \cite{B}).
The first inequality from the left side in \eqref{E7} follows immediately from \eqref{E3} and \eqref{E11}. The last inequality follows from Theorem \ref{T1}.
\end{proof}
\begin{corollary}\label{C2} Let $A ,\;B \in P_n$ then for  $0\leq \alpha \leq  1$
\begin{enumerate}[\rm{(}i{\rm)}]
  \item $|||(A+B-|A-B|)_+|||\leq 2 |||A^{\alpha} B^{1-\alpha}|||$
  \item $|||(A+B-|A-B|)_-|||\leq 2 |||A^{\alpha} B^{1-\alpha}|||.$
\end{enumerate}
\end{corollary}
\begin{proof} (i) As $A ,\;B \in P_n,$ hence, without loss of generality we assume
\begin{eqnarray}\lambda_1(A+B-|A-B|)&\geq & \lambda_2(A+B-|A-B|)\geq\cdots\geq\lambda_k(A+B-|A-B|)\geq 0\nonumber\\
&>&\lambda_{k+1}(A+B-|A-B|)\geq \cdots\geq \lambda_n(A+B-|A-B|)\nonumber\end{eqnarray}
and
\begin{eqnarray}\lambda_1(A^{\alpha/2}B^{1-\alpha}A^{\alpha/2})\geq  \lambda_2(A^{\alpha/2}B^{1-\alpha}A^{\alpha/2})\geq\cdots\geq\lambda_n(A^{\alpha/2}B^{1-\alpha}A^{\alpha/2})\geq 0.\nonumber\end{eqnarray}
The matrix  $A+B-|A-B|$ is Hermitian, so unitarily diagonalizable, i.e.,
$$A+B-|A-B|= W^*D_2W,$$ for $W$ a unitary matrix and $D_2$ a diagonal matrix given by
$$D_2=diag(\lambda_1(A+B-|A-B|),\cdots,\lambda_n(A+B-|A-B|)).$$
Now, using Jordan decomposition of $A+B-|A-B|,$ (see \cite{B}, page 99) provides that
$$(A+B-|A-B|)_+= W^*D_{2_+}W\;\;\;\; {\rm{and}}\;\;\; (A+B-|A-B|)_-= W^*D_{2_-}W,$$
where $D_{2_+}$ and $D_{2_-}$ are diagonal matrices in $P_n,$  given by
$$D_{2_+}=diag(\lambda_1(A+B-|A-B|),\cdots,\lambda_k(A+B-|A-B|),0,\cdots 0)$$ and
$$D_{2_-}=diag(0,\cdots,-\lambda_{k+1}(A+B-|A-B|),\cdots,-\lambda_n(A+B-|A-B|)).$$
By the above discussion, we clearly obtain
\begin{eqnarray}\lambda_i(A+B-|A-B|)_+ = \left \{\begin{array}{c}
                                            \lambda_i(A+B-|A-B|),\; \;\;\;\;\;\;for\; i=1,2,\cdots,k \\
                                            \;\;\;\;\;\;\;0,\;\;\;\;\;\;\;\;\;\;\;\;\;\;\;\;\;\;  for\;\; i=k+1,k+2,\cdots,n,
                                          \end{array}\right . \nonumber\end{eqnarray}
and
\begin{eqnarray}\lambda_i(A+B-|A-B|)_- = \left \{\begin{array}{c}
                                             \;\;\;\;\;\;\;0, \;\;\;\;\;\;\;\;\;\;\;\;\;\;\;\;\;\;\;\;\;\;\;\;\;\;\;\;\;\;\;\;\;\;\;\;\;\;\;for\; i=1,2,\cdots,k \\
                                            -\lambda_i(A+B-|A-B|),\;\; for\; i=k+1,k+2,\cdots,n.
                                          \end{array}\right . \nonumber\end{eqnarray}

Now,  using inequality \eqref{E2} alongwith $\lambda_i(A^{\alpha/2} B^{1-\alpha}A^{\alpha/2})=\lambda_i(A^{\alpha} B^{1-\alpha}),$  we obtain
\begin{equation}\label{E9} \lambda_i((A+B-|A-B|)_+)\leq 2\lambda_i(A^{\alpha/2} B^{1-\alpha}A^{\alpha/2}),\;\;\;\; for\;\;\;i=1,2,\cdots,n,\nonumber \end{equation}
i.e,
\begin{equation}\label{SE9} s_i((A+B-|A-B|)_+)\leq 2s_i(A^{\alpha/2} B^{1-\alpha}A^{\alpha/2}),\;\;\;\; for\;\;\;i=1,2,\cdots,n. \end{equation}

On using Theorem IV.2.2  and then Proposition IX.1.1 of \cite{B} in \eqref{SE9}, we obtain
\begin{eqnarray}\label{E10} |||(A+B-|A-B|)_+||| &\leq& 2|||A^{\alpha/2} B^{1-\alpha}A^{\alpha/2}|||\nonumber\\
&\leq& 2 |||A^{\alpha} B^{1-\alpha}|||.\end{eqnarray}
This completes the proof of (i).\\
For a proof of (ii), use \eqref{E3} and \eqref{E11} to obtain,
\begin{eqnarray}\label{E12}\lambda_i{((A+B-|A-B|)_-})\leq 2 \lambda_i(B).\end{eqnarray}
Now, replace $B$ by $A$ in \eqref{E12}, we obtain,
\begin{eqnarray}\label{E13}\lambda_i{((A+B-|A-B|)_-})\leq 2 \lambda_i(A).\end{eqnarray}
Again, on using similar technique as in Theorem \ref{T1}, we get the desired result.
\end{proof}

The following corollary is an immediate consequence of  Corollary \ref{C2}.

\begin{corollary}\label{C3} Let $A ,\;B \in P_n$ then for  $0\leq \alpha \leq  1$
\begin{equation}\label{xC3}|| A+B-|A-B|\; ||\leq 2 ||A^{\alpha} B^{1-\alpha}||.
\end{equation}
\end{corollary}
\begin{proof} The operator norm for any Hermitian matrix $T$ is given by
$$||T||=max\{||T_+||,||T_-||\}.$$  Using the above fact for the matrix $A+B-|A-B|$ and Corollary \ref{C2} to obtain \eqref{xC3}.

\end{proof}
\begin{theorem}\label{T2} Let $A ,\;B \in P_n$ then for  $0\leq \alpha \leq  1,$  some projection $P$ and $\beta \geq 0,$
\begin{equation}
|||A+B-|A-B||||\leq 2 |||A^{\alpha} B^{1-\alpha}-\beta A^{\alpha/2} P A^{-\alpha/2}|||.
\end{equation}
\end{theorem}
\begin{proof} Let $X=diag(x_1,x_2,\ldots,x_n)$ and $T=diag(t_1,t_2,\ldots,t_n)$ be the matrices comprised of $x_i's$ and $t_i's$ as  eigenvalues of $A+B-|A-B|$ and $2 A^{\alpha/2} B^{1-\alpha} A^{\alpha/2}$ in decreasing order respectively. Using Theorem (2.2) on $X$ and $T,$ we get $$x_i\leq t_i\;\;\; \emph{for}\;\;\; i=1,2,\ldots,n. $$
If $\beta= Tr(T)-Tr(X),$ then on using Corollary \ref{C1}, we obtain $\beta \geq 0.$ Consider
$$ T_1= 2 A^{\alpha/2} B^{1-\alpha}A^{\alpha/2}-\beta Q_n,$$
where $\sum\limits_{i=1}^{n} t_i Q_i$ is the spectral decomposition of $2 A^{\alpha/2}B^{1-\alpha}A^{\alpha/2}.$ It is clear from the construction of $T_1$ that eigenvalues of $T_1$ are all same and in the same order as that of $2 A^{\alpha/2}B^{1-\alpha}A^{\alpha/2}$ except the last one. So we may assume $ (t_1,t_2,\ldots,t_{n-1},\gamma_n)^t$ as a column vector of eigenvalues of $T_1,$ satisfying $$ \sum\limits_{i=1}^{k} x_i \leq \sum\limits_{i=1}^{k} t_i\;\;\; \emph{for}\;\;\; k=1,2,3,\ldots,n-1,   $$ and $$\sum\limits_{i=1}^{n} x_i = \sum\limits_{i=1}^{n-1} t_i + \gamma_n.$$
Finally, using Example II.3.5 in \cite{B}, we get
$$ \sum\limits_{i=1}^{k} |x_i| \leq \sum\limits_{i=1}^{k} t_i\;\;\; \emph{for}\;\;\; k=1,2,\ldots,n-1,   $$ and $$\sum\limits_{i=1}^{n} |x_i|\leq \sum\limits_{i=1}^{n-1} t_i + |\gamma_n|.$$
Equivalently,
$$ \sum\limits_{i=1}^{k} s_i(A+B-|A-B|) \leq \sum\limits_{i=1}^{k} s_i(T_1)\;\;\; \emph{for}\;\;\; k=1,2,\ldots,n. $$
Hence,
\begin{eqnarray}\label{E13}|||A+B-|A-B||||&\leq&  |||2A^{\alpha/2} B^{1-\alpha}A^{\alpha/2}- \beta Q_n |||\nonumber\\
&=& |||A^{-\alpha/2}(2A^{\alpha} B^{1-\alpha}A^{\alpha/2}- \beta A^{\alpha/2} Q_n) |||\nonumber\\
&\leq&  ||| 2A^{\alpha}B^{1-\alpha}-\beta A^{\alpha/2}  Q_n A^{-\alpha/2}|||,\nonumber
\end{eqnarray}
using Theorem IV.2.2 of \cite{B} for the first inequality and Proposition IX.1.1 of \cite{B} for the second inequality. This completes the proof.

\end{proof}
\textbf{Acknowledgements:}
The authors would like to sincerely thank the referee for several useful comments improving
the paper.

\end{document}